\documentclass{amsart}
\usepackage{amsmath,amsfonts,amsthm}
\usepackage[english]{babel}

\newtheorem{theorem}{Theorem}[section]
\newtheorem{lemma}[theorem]{Lemma}
\newtheorem{definition}[theorem]{Definition}
\newtheorem{proposition}[theorem]{Proposition}
\newtheorem{corollary}[theorem]{Corollary}
\newtheorem{remark}[theorem]{Remark}
\newtheorem{question}[theorem]{Question}
\newtheorem{problem}[theorem]{Problem}
\newtheorem{example}[theorem]{Example}

\numberwithin{equation}{section}

\begin{document}

\title{Equiconnected spaces and Baire classification of separately continuous functions and their analogs}

\author{Olena Karlova}

\author{Volodymyr Maslyuchenko}

\author{Volodymyr Mykhaylyuk}

\begin{abstract}
 We investigate the Baire classification of mappings $f:X\times Y\to Z$, where $X$ belongs to a wide class of spaces, which includes all metrizable spaces,  $Y$ is a topological space, $Z$ is an equiconnected space, which are continuous in the first variable and for a dense set in $X$ these mappings are functions of a Baire class $\alpha$ in the second variable.
\end{abstract}


\maketitle

\section{Introduction}

We will denote by $P(X,Y)$ the collection of all mappings $f:X\to Y$ with a property $P$. For a mapping $f:X\times Y\to Z$ and a point $(x,y)\in X\times Y$ write $f^x(y)=f_y(x)=f(x,y)$.

Let $P$ and $Q$ be some properties of mappings. Set
$$
X_Q(f)=\{x\in X: f^x\in Q(Y,Z)\};\quad Y_P(f)=\{y\in Y: f_y\in
P(X,Z)\};
$$
$$
PQ(X\times Y,Z)=\{f\in Z^{X\times Y}:X_Q(f)=X, Y_P(f)=Y\};
$$
$$
P\overline{Q}(X\times Y,Z)=\{f\in Z^{X\times Y}:\overline{X_Q(f)}=X, Y_P(f)=Y\}.
$$

Let $C(X,Y)$ denote the collection of all continuous mappings between $X$ and $Y$.
A mapping $f:X\times Y\to Z$ is
{\it separately continuous} if $f\in CC(X\times Y,Z)$. We call a mapping $f:X\times Y\to Z$
  {\it vertically nearly separately continuous} if $f\in C\overline{C}(X\times Y,Z)$.

A mapping $f:X\to Y$ is said to be {\it a mapping of the first Baire class} or {\it a Baire-one mapping} if there exists a sequence of continuous mappings  \mbox{$f_n:X\to Y$} such that
$f_n(x)\to f(x)$ for every $x\in X$. The class of all Baire-one mappings $f:X\to Y$ will be denoted by $B_1(X,Y)$. Let $\alpha>0$ be an at most countable ordinal and assume that classes $B_\xi(X,Y)$  are already defined for all $\xi<\alpha$. A mapping $f:X\to Y$ {\it belongs to the $\alpha$-th Baire class}, $f\in B_\alpha(X,Y)$, if  there exists a sequence of mappings $f_n\in \bigcup\limits_{\xi<\alpha}B_{\xi}(X,Y)$ such that $f_n(x)\to f(x)$ for every $x\in X$. Notice that for an arbitrary sequence $(\alpha_n)_{n=1}^{\infty}$ of ordinals $\alpha_n<\alpha$ such that $\lim\limits_{n\to\infty}(\alpha_n+1)=\alpha$, we can choose a sequence $(f_n)_{n=1}^{\infty}$ in such a way that $f_n\in B_{\alpha_n}(X,Y)$ for every $n\in\mathbb N$. Define $B_{0}(X,Y)=C(X,Y)$.

 In 1898 H.~Lebesgue \cite{Leb1} proved that every real-valued separately continuous functions of two real variables is of the first Baire class. In honour of his theorem,  we  call a collection $(X,Y,Z)$  of topological spaces
   {\it a Lebesgue $\alpha$-triple}, if $CB_\alpha(X\times Y,Z)\subseteq B_{\alpha+1}(X\times Y,Z)$, where $0\le \alpha<\omega_1$.

Lebesgue's theorem was generalized by many mathematicians (see \cite{Hahn, Moran, Ru, Vera, SobPP, BanakhT, Bu1, KMas3, K4, M} and the references given there). The classical work here is the paper of W.~Rudin  \cite{Ru}, who proved the following theorem.

\begin{theorem}\label{A}
Let  $X$ be a metrizable space, $Y$ a topological space and $Z$ a locally convex space. Then $C\overline{C}(X\times Y,Z)\subseteq B_{1}(X\times Y,Z)$.
\end{theorem}

A collection $(X,Y,Z)$ of topological spaces is called {\it a Rudin $\alpha$-triple}, if $C\overline{B}_\alpha(X\times Y,Z)\subseteq B_{\alpha+1}(X\times Y,Z)$, where $0\le \alpha<\omega_1$.

The following question is still unanswered.
\begin{problem}
Do there exist a metrizable space $X$, a topological space $Y$ and a topological vector space (or, more general, an equiconnected space) $Z$ such that $(X,Y,Z)$ is not a Rudin $0$-triple?
\end{problem}

The following result was proved by V.~Maslyuchenko and A.~Kalancha \cite{KMas3} .

\begin{theorem}\label{B}
  Let $X$ be a metrizable space with finite \v{C}ech-Lebesgue dimension, $Y$ a topological space, $Z$ a topological vector space and $0\le\alpha<\omega_1$. Then $(X,Y,Z)$ is a Rudin $\alpha$-triple.
\end{theorem}

In \cite{SobPP} O.~Sobchuk introduced the class of PP-spaces (see definition in Section~\ref{sec:PP}), which includes the class of all metrizable spaces, and obtained the following fact (see \cite{Sob2}).
\begin{theorem}\label{C}
  A collection $(X,Y,Z)$ is a Lebesgue $\alpha$-triple if $X$ is a PP-space, $Y$ is a topological space and $Z$ is a locally convex space.
\end{theorem}

The following two results were established by  T.~Banakh \cite{BanakhT}, who introduced  metrically quarter-stratifiable spaces and studied their applications to Baire classification of separately continuous functions with values in equiconnected spaces (see definitions in Section~\ref{sec:1} and Section~\ref{sec:PP}).

\begin{theorem}\label{D}
Let  $X$ be a metrically quarter-stratifiable space, $Y$ a topological space
and $Z$ an equiconnected space. If, moreover, $X$ is  paracompact and strongly countable-dimensional  or $Z$ is locally convex, then  $(X,Y,Z)$ is a Lebesgue $0$-triple.
\end{theorem}

\begin{theorem}\label{E}
 Let $1\le \alpha<\omega_1$, $X$ a metrically quarter-stratifiable space, $Y$ a topological space and $Z$ a contractible space. Then
 $(X,Y,Z)$ is a Lebesgue $\alpha$-triple.
\end{theorem}

Obviously, every Rudin $\alpha$-triple is a Lebesgue $\alpha$-triple. In \cite[Example~5.7]{BanakhT} was given an example which shows that Theorems~\ref{C} and \ref{D} cannot be generalized to the case when $(X,Y,Z)$ is a Rudin $0$-triple. Therefore, it is naturally to ask if we can replace Lebesgue triple with a Rudin triple in Theorems~\ref{C} and \ref{E}, and under which assumptions on $X$ analogs of Theorems~\ref{B}, \ref{C}, \ref{D} and \ref{E} take place? To answer these questions, we first introduced notions of a convex combination and an $\lambda$-sum of elements of an equiconnected space (see Sections~\ref{sec:1} and \ref{sec:2}) and describe some of their properties. In Section~\ref{sec:PP} we introduce the class of strong PP-spaces and find out that  $(X,Y,Z)$ is a Rudin $\alpha$-triple for $\alpha\ge 0$ if $X$ is a strong PP-space, $Y$ is a topological space and  $Z$ is a locally convex equiconnected space. Further, in Section~\ref{sec:3} we prove that $(X,Y,Z)$ is  a Rudin $\alpha$-triple for $\alpha\ge 0$ if $X$ is a paracompact Hausdorff strongly countable-dimensional strong  PP-space, $Y$ is a topological space and $Z$ is an equiconnected space. This gives the generalization of Theorem~\ref{B} for equiconnected space $Z$ and for Rudin triples. Functions of the classes $C\overline{B}_\alpha$ with $\alpha>0$, are investigated in Section~\ref{sec:4}. There we show that if $X$ is a strong PP-space, $Y$ is a topological space and $Z$ is a conractible space, then $(X,Y,Z)$ is a Rudin $\alpha$-triple. Finally, in Section~\ref{sec:5} we give several examples of PP-spaces which are not strong PP-spaces, consequently, these examples show that Theorem~\ref{E} cannot be generalized to Rudin triples.

\section{Convex combinations in an equiconnected space and their properties}\label{sec:1}

\begin{definition}\label{def:4}{\rm
  {\it An equiconnected space} is a pair $(X,\lambda)$, consisting of a topological space $X$ and a continuous function $\lambda:X\times X\times [0,1]\to X$ such that

{ (i)} $\lambda(x,y,0)=x$;

{ (ii)} $\lambda(x,y,1)=y$;

{ (iii)} $\lambda(x,x,t)=x$\\ for all $x,y\in X$ and $t\in [0,1]$.}
\end{definition}
\noindent The simplest example of an equiconnected space is a convex subset of a topological vector space, where  $\lambda$ can be defined as  $\lambda(x,y,t)=(1-t)x+t y$.

For every $n\in\mathbb N$ denote $$S_n=\{\alpha=(\alpha_1,\dots ,\alpha_n)\in\mathbb R^n: \,\,\alpha_1,\dots ,\alpha_n\geq 0,\,\,\alpha_1+\dots +\alpha_n=1\}.$$

Let $(X,\lambda)$ be an equiconnected space. By induction in $n\in\mathbb N$ we define a sequence $(\lambda_n)_{n=1}^{\infty}$ of mappings $\lambda_n:X^n\times S_n\to X$.

For $n=1$ set $\lambda_1(x,1)=x$ for every $x\in X$. Now for  $n\in\mathbb N$, $x_1, \dots x_{n+1}\in X$ and $(\alpha_1,\dots ,\alpha_{n+1})\in S_{n+1}$ let
\begin{gather*}
\lambda_{n+1}(x_1, \dots x_{n+1},\alpha_1,\dots ,\alpha_{n+1})=\\=\lambda_n(\lambda(x_1,x_2, \frac{\alpha_2}{\alpha_1+\alpha_2}), x_3, \dots x_{n+1},\alpha_1+\alpha_2, \alpha_3 \dots ,\alpha_{n+1}), \mbox{\,\,if\,\,} \alpha_1+\alpha_2>0;   \\
\lambda_{n+1}(x_1, \dots x_{n+1},\alpha_1,\dots ,\alpha_{n+1})=\\=\lambda_n(x_2, x_3, \dots x_{n+1},\alpha_2, \alpha_3 \dots ,\alpha_{n+1}), \mbox{\,\,if\,\,} \alpha_1+\alpha_2=0.
\end{gather*}

\begin{definition}\label{def:5}{\rm For $n\in\mathbb N$, $(\alpha_1,\dots,\alpha_n)\in S_n$ and any elements $x_1,\dots x_n$ of an equiconnected space $(X,\lambda)$ the element $\lambda_n(x_1,\dots x_n,\alpha_1,\dots,\alpha_n)$ is called {\it a convex combination of the elements $x_1,\dots x_n$ with the coefficients $\alpha_1,\dots,\alpha_n$}.}
\end{definition}

\begin{proposition}\label{pr:2.1}
  Let $(X,\lambda)$ be an equiconnected space. Then $(\lambda_n)_{n=1}^{\infty}$ has  the following properties:

 (i) if  $n\ge 2$, $1\leq k\leq n$ and $\alpha_k=0$, then
 \begin{gather*}
   \lambda_{n}(x_1, \dots x_{n},\alpha_1,\dots ,\alpha_{n})=\\=\lambda_{n-1}(x_1, \dots x_{k-1}, x_{k+1}\dots x_{n},\alpha_1, \dots \alpha_{k-1}, \alpha_{k+1}\dots \alpha_{n})
 \end{gather*}
 \qquad\,\, for any $x_1,\dots x_n\in X$ and $(\alpha_1,\dots ,\alpha_n)\in S_n$;

 (ii) $\lambda_{n}(x, \dots x,\alpha_1,\dots ,\alpha_{n})=x$ for any $x\in X$ and $(\alpha_1,\dots ,\alpha_n)\in S_n$;

(iii) every $\lambda_n$ is continuous.
\end{proposition}

\begin{proof}
By induction on $n$ properties $(i)$ and $(ii)$ easily follow from the definition of functions $\lambda_n$.

$(iii)$. Clearly, $\lambda_1$ is continuous. Assume that functions $\lambda_1,\dots,\lambda_n$ are continuous and show that $\lambda_{n+1}$ is continuous.

Fix points $x_1,\dots, x_{n+1}\in X$ and a vector $(\alpha_1,\dots, \alpha_{n+1})\in S_{n+1}$, and prove that $\lambda_{n+1}$ is continuous at $(x_1,\dots, x_{n+1},\alpha_1,\dots, \alpha_{n+1})$.

Let
\begin{gather*}
G=\{(y_1,\dots y_{n+1},\beta_1,\dots \beta_{n+1}): y_1,\dots y_{n+1}\in X,\phantom{,,,,,,,,,,,,,,,,,,,,,}\\\phantom{,,,,,,,,,,,,,,,,,,,,,,,,,,,,} (\beta_1,\dots \beta_{n+1})\in S_{n+1},\,\beta_1+\beta_2>0 \}.
\end{gather*}
Since $G$ is open and functions $\lambda$ and $\lambda_n$ are continuous, $\lambda_{n+1}$ is continuous at every point of $G$, in particular, $\lambda_{n+1}$ is continuous at
$(x_1,\dots, x_{n+1},\alpha_1,\dots, \alpha_{n+1})$ if $\alpha_1+\alpha_2>0$.

Now let $\alpha_1=\alpha_2=0$, $z=\lambda_{n+1}(x_1,\dots, x_{n+1},\alpha_1,\dots, \alpha_{n+1})$ and let $W$ be an open neighborhood of $z$ in $X$. Consider the set $A=\{\lambda(x_1,x_2,\alpha):0\leq \alpha\leq 1\}$. Observe that $A$ is compact and  $\lambda_n(a,x_3,\dots, x_{n+1},0,\alpha_3,\dots, \alpha_{n+1})=z$ for every $a\in A$ by $(i)$. Since $\lambda_n$ is continuous, there exist an open set $U$ in $X$ with $A\subseteq U$, open neighborhoods $U_3, \dots, U_{n+1}$ in $X$ of points $x_3,\dots, x_{n+1}$, respectively, and open neighborhoods $V, V_3,\dots,V_{n+1}$ of points $0,\alpha_3,\dots, \alpha_{n+1}$ in $[0,1]$ such that $\lambda_{n}(y,z_3, \dots z_{n+1},\gamma,\beta_3,\dots, \beta_{n+1})\in W$ for any $y\in U$, \mbox{$z_3\in U_3$},..., $z_{n+1}\in U_{n+1}$ and $\gamma\in V,\beta_3\in V_3,\dots, \beta_{n+1}\in V_{n+1}$. Now, taking into account that $A\subseteq U$ and that $\lambda_2$ is continuous, we choose neighborhoods $U_1$ and $U_2$ of $x_1$ and $x_2$ such that $\lambda(z_1,z_2,\beta)\in U$ for any $z_1\in U_1$, $z_2\in U_2$ and $\beta\in [0,1]$. Moreover, choose neighborhoods $V_1$ and $V_2$ of $0$ in $[0,1]$ such that $V_1+V_2\subseteq V$. Show that
$$\lambda_{n+1}(z_1, \dots z_{n+1},\beta_1,\dots ,\beta_{n+1})\in W$$
for any $z_1\in U_1,\dots z_{n+1}\in U_{n+1}$ and $\beta_1\in V_1, \dots, \beta_{n+1}\in V_{n+1}$.

If $\beta_1=\beta_2=0$, then, taking into account that $z_2\in U$ and $0\in V$, we conclude that
$$\lambda_{n+1}(z_1, \dots z_{n+1},\beta_1,\dots ,\beta_{n+1})=\lambda_{n}(z_2,z_3 \dots z_{n+1},0,\beta_3,\dots ,\beta_{n+1})\in W.$$
If $\beta_1+\beta_2>0$, then
$$\lambda_{n+1}(z_1, \dots z_{n+1},\beta_1,\dots ,\beta_{n+1})=\lambda_{n}(y,z_3 \dots z_{n+1},\gamma,\beta_3,\dots ,\beta_{n+1})\in W,$$
where $y=\lambda(z_1,z_2, \frac{\beta_2}{\beta_1+\beta_2})\in U$ and $\gamma=\beta_1+\beta_2\in V_1+V_2\subseteq V$.\hfill$\Box$
\end{proof}

Let $(X,\lambda)$ be an equiconnected space and $A\subseteq X$ a nonempty set. Put $\lambda^0(A)=A$, $\lambda^n(A)=\lambda(\lambda^{n-1}(A)\times A\times [0,1])$ for $n\in\mathbb N$ and $\lambda^\infty(A)=\bigcup\limits_{n=1}^\infty \lambda^n(A)$. According to \cite{BanakhT} we call an equiconnected space $(X,\lambda)$ {\it locally convex} if for every $x\in X$ and a neighborhood $U$ of $x$ in $X$ there exists a neighborhood $V$ of $x$ such that $\lambda^\infty(V)\subseteq U$.

The following fact directly follows from the definitions.

\begin{proposition}\label{pr:2.2}
   Let$(X,\lambda)$ be an equiconnected space, $A\subseteq X$ and $n\in\mathbb N$. Then $$\lambda^n(A)=\{\lambda_n(x_1, \dots x_{n},\alpha_1,\dots,\alpha_{n}):(x_1, \dots x_{n})\in A^n, (\alpha_1,\dots ,\alpha_{n})\in S_n\}.$$
\end{proposition}

\section{$\lambda$-sums and Baire measurable functions}\label{sec:2}

\begin{definition}\label{def:6}{\rm Let $(I,\leq)$ be a well ordered set, $(x_i)_{i\in I}$ a family of elements $x_i$ of an equiconnected space $(X,\lambda)$, and $(\alpha_i)_{i\in I}$ a collection of scalars $\alpha_i\geq 0$ such that:

$(1)$ $\{i\in I: \alpha_i\ne 0\}=\{i_k: 1\leq k\leq n\}$;

$(2)$ $i_1 < i_2 <\dots < i_n$;

$(3)$ $\alpha_{i_1}+\alpha_{i_2}+\dots +\alpha_{i_n}=1$.

\noindent Then the element $\lambda_n(x_{i_1},\dots , x_{i_n}, \alpha_{i_1},\dots ,\alpha_{i_n})$ is called  {\it a $\lambda$-sum of elements  $(x_i)_{i\in I}$ with coefficients $(\alpha_i)_{i\in I}$} and is denoted by  ${\sum\limits_{i\in I}}^\lambda \alpha_i x_i$.}\end{definition}

Let us observe that ${\sum\limits_{i\in I}}^\lambda \alpha_i x_i={\sum\limits_{i\in I}} \alpha_i x_i$ for an arbitrary topological vector space $X$, where $\lambda(x,y,t)=(1-t)x+t y$.

\begin{theorem}\label{th:2.3}
   Let $X$ and $Y$ be topological spaces, $(Z,\lambda)$ an equiconnected space, $(I,\leq)$ a well ordered set, let $0\le \alpha<\omega_1$, $(f_i)_{i\in I}$ a family of mappings {$f_i:Y\to Z$} of Baire class $\alpha$ and $(\varphi_i)_{i\in I}$ a locally finite partition of unity on $X$. Then a mapping $f:X\times Y\to Z$, $f(x,y)={\sum\limits_{i\in I}}^\lambda \varphi_i(x) f_i(y)$, is of the Baire class $\alpha$; in particular, if all mappings $f_i$ are continuous, then so is $f$.
\end{theorem}

\begin{proof} We will argue by induction in $\alpha$. We first consider the case $\alpha=0$, i.e. the case when all mappings $f_i$ are continuous.
Fix a point $x_0\in X$ and show that $f$ is continuous at every point of $\{x_0\}\times Y$.

Using locally finiteness of partition of unity $(\varphi_i)_{i\in I}$, we choose an open neighborhood $U$ of $x_0$ in $X$ such that $I_0=\{i\in I: {\rm supp}\,\varphi_n \cap U \ne\O\}$ is finite. Let $I_0=\{i_k:1\leq k\leq n\}$ and  $i_1 < i_2 <\dots < i_n$. The definition of $\lambda$-sum and Proposition~\ref{pr:2.1} $(i)$ imply
$$f(x,y)=\lambda_n(f_{i_1}(y),\dots , f_{i_n}(y), \varphi_{i_1}(x),\dots ,\varphi_{i_n}(x))$$
for any $x\in U$ and $y\in Y$. Since mappings $\lambda_n$, $f_i$ and $\varphi_i$ are continuous, $f$ is continuous on $U\times Y$; in particular, $f$ is continuous at every point of $\{x_0\}\times Y$. Hence, the theorem is true for $\alpha=0$.

Now suppose that the theorem is true for all ordinals  $\alpha<\beta$ with some fixed ordinal $0<\beta<\omega_1$. Let $(f_i)_{i\in I}$ be a family of mappings $f_i\in B_\beta(Y,Z)$. Choose an increasing sequence $(\alpha_n)_{n=1}^\infty$ of ordinals $\alpha_n<\beta$ such that  $\lim\limits_{n\to\infty}(\alpha_n+1)=\beta$ and a sequence of families $(f_{i,n})_{i\in I}$ of mappings $f_{i,n}\in B_{\alpha_n}(Y,Z)$ such that $f_{i,n}(y)\to f_i(y)$ for all $y\in Y$ and $i\in I$. By the assumption, the mapping $g_n:X\times Y\to Z$, $g_n(x,y)={\sum\limits_{i\in I}}^\lambda \varphi_i(x) f_{i,n}(y)$, is of the class $\alpha_n$ for every  $n\in\mathbb N$. Moreover, the continuity of $\lambda_n$ implies  $g_n(x,y)\to f(x,y)$ for all $x\in X$ and $y\in Y$.\hfill$\Box$

\end{proof}

\section{Strong PP-spaces and mappings into locally convex spaces}\label{sec:PP}

\begin{definition}{\rm
A topological space $(X,{\mathcal T})$ is called {\it quarter-stratifiable} if there exists a function $g:{\mathbb N}\times X\to\mathcal T$ (called {\it quarter-stratifying function}) such that
\begin{enumerate}
  \item $X=\bigcup\limits_{x\in X}g(n,x)$ for every $n\in\mathbb N$;

  \item ($x\in g(n,x_n)$, $n\in\mathbb N$) $\Rightarrow$ ($x_n\to x$).
\end{enumerate}

A topology ${\mathcal T}'$ on a topological space $(X,{\mathcal T})$ is called  {\it quarter-stratifying} if there exists a quarter-stratifying function $g:{\mathbb N}\times X\to\mathcal T$ for $X$ such that $g({\mathbb  N}\times X)\subseteq {\mathcal T}'$.

A topological space $X$ is said to be {\it metrically quarter-stratifiable} if it admits a weaker metrizable quarter-stratifying topology.
}
\end{definition}

The topology on a topological space $X$ we denote by $\mathcal T(X)$.

\begin{definition}\label{def:2} {\rm
  A topological space $X$ is called {\it (strong) PP-space} if (for any dense set $D\subseteq X$) there exist a sequence $({\mathcal U}_n)_{n=1}^\infty$ of $((\varphi_{i,n}:i\in
  I_n))_{n=1}^\infty$ of locally finite partitions of unity on $X$  and a sequence $((x_{i,n}:i\in I_n))_{n=1}^\infty$ of families of points of $X$ (of $D$) such that
      \begin{equation*}
       (\forall x\in X) (\forall U\ni x, U\in{\mathcal T}(X)) (\exists n_0\in\mathbb N)(\forall n\ge n_0) (\forall i\in I_n)
       \end{equation*}
       \begin{equation}\label{eq:3}
    (x\in {\rm supp}\varphi_{i,n} \Longrightarrow x_{i,n}\in U)
    \end{equation}
    }
\end{definition}
    Obviously, every strong PP-space is a PP-space.

\begin{remark}{\rm $\,$
  \begin{enumerate}
    \item Every metrically quarter-stratifiable space is a PP-space according to \cite[Remark 3.1]{M}.

    \item Every Hausdorff PP-space is metrically quarter stratifiable by \cite[Proposition 3.2]{M}.
  \end{enumerate}  }
\end{remark}

It is easy to see that any metrizable space, or, more general, every space equipped with a topology generated by a pseudo-metric, is strong PP-space. Therefore, a topological space $X$ equipped with the topology generated by a pseudo-metric which is not a metric on $X$, is a PP-space, but is not metrically quarter-stratifiable.

The following proposition gives an example of a non-metrizable strong PP-space.

\begin{proposition}\label{th:5.1} The Sorgenfrey line $\mathbb L$ is a strong PP-space. \end{proposition}

\begin{proof} Let $D\subseteq \mathbb L$ be a dense set. For any $n\in\mathbb N$ and $i\in \mathbb Z$ denote by  $\varphi_{i,n}$ the characteristic function of $[\frac{i-1}{n},\frac{i}{n})$ and choose a point $x_{i,n}\in [\frac{i}{n},\frac{i+1}{n})\cap D$. Then the sequences  $\bigl((\varphi_{i,n}:i\in I_n) \bigr)_{n=1}^\infty$ and $\bigl((x_{i,n}:i\in I_n) \bigr)_{n=1}^\infty$ satisfy (\ref{eq:3}).\hfill$\Box$

\end{proof}

 Examples of PP-spaces which are not strong PP-spaces can be found in Section~\ref{sec:5} (see Examples~\ref{ex:1} and \ref{ex:2}).

\begin{theorem}\label{th:4.1} Let $X$ be a strong PP-space, $Y$ a topological space, $(Z,\lambda)$ a locally convex equiconnected space and  $0\le\alpha<\omega_1$. Then $(X,Y,Z)$ is a Rudin  $\alpha$-triple. \end{theorem}

\begin{proof} Let $f\in C\overline{B}_\alpha (X\times Y, Z)$. Since $X_{B_\alpha}(f)$ is dense in $X$, we can choose a sequence
$((\varphi_{i,n}:i\in I_n))_{n=1}^\infty$ of locally finite partitions of unity on  $X$ and a sequence $((x_{i,n}:i\in I_n))_{n=1}^\infty$ of families of points of
$X_{B_\alpha}(f)$ which satisfy~(\ref{eq:3}).

For every $n\in {\mathbb N}$ we well order the set $I_n$ and for all $(x,y)\in X\times Y$ we set $$ f_n(x,y)={\sum\limits_{i\in I_n}}^\lambda\varphi_{i,n}(x)f(x_{i,n},y). $$ Then $f^{x_{i,n}}\in B_\alpha(Y,Z)$ for every pair $(i,n)\in I_n \times {\mathbb N}$. Therefore, applying Theorem~\ref{th:2.3}, we have that $f_n\in {B}_\alpha
(X\times Y, Z)$ for every $n\in\mathbb N$.

We will show that $f_n(x,y)\to f(x,y)$ on $X\times Y$. Fix $(x_0,y_0)\in X\times Y$ and consider an arbitrary neighborhood $W$ of $z_0=f(x_0,y_0)$ in $Z$. Since $Z$ is locally convex, there is a neighborhood $W_1$ of $z_0$ such that $\lambda^\infty(W_1)\subseteq W$. The continuity of $f$ in the first variable at $(x_0,y_0)$  implies the existence of a neighborhood $U$ of  $x_0$ in $X$ such that $f(x,y_0)\in W_1$ for every  $x\in U$. Now, using~(\ref{eq:3}), we choose a number $n_0$ such that $x_{i,n}\in U$ for all $n\geq n_0$  and $i\in I_n$  with $x_0\in {\rm supp}\varphi_{i,n}$.

Now we prove that $f_n(x_0,y_0)\in W$ for all $n\geq n_0$. Fix $n\geq n_0$ and set $$I=\{i\in I_n:x_0\in {\rm supp}\varphi_{i,n}\}.$$
Let $I=\{i_1, i_2,\dots i_k\}$, where $i_1< i_2<\dots < i_k$. Denote $\alpha_1=\varphi_{i_1,n}(x_0)$,$\dots$,$\alpha_k=\varphi_{i_k,n}(x_0)$, $z_1=f(x_{i_1,n},y_0),\dots , z_k=f(x_{i_k,n},y_0)$. Then $z_1,\dots z_k \in W_1$ and $$f_n(x_0,y_0)=\lambda_k(z_1,\dots , z_k,\alpha_1,\dots \alpha_k)\in \lambda^k(W_1)\subseteq W.$$\hfill$\Box$
\end{proof}

\begin{corollary}\label{cor:4.4} Let $X$ be a metrizable space, $Y$ a topological space, $(Z,\lambda)$ a locally convex equiconnected space. Then  $(X,Y,Z)$ is a Rudin $\alpha$-triple for every $0\le \alpha<\omega_1$. \end{corollary}

\section{Mappings into equiconnected spaces}\label{sec:3}

\begin{definition}\label{def:3} {\rm
  A sequence $(\mathcal A_n)_{n=1}^{\infty}$ of families $\mathcal A_n=(A_{i,n}:i\in I_n)$ of subsets $A_{i,n}$ of a topological space  $X$ is called {\it uniformly pointwise finite in $X$} if for every $x\in X$ there exists $k_x\in \mathbb N$ such that for every $n\in \mathbb N$ the set  $\{i\in I_n: x\in A_{i,n}\}$ contains at most $k_x$ elements.}
\end{definition}

\begin{theorem}\label{th:4.2}
  Let $X$ and $Y$ be topological spaces, $(Z,\lambda)$ an equiconnected space, $f:X\times Y\to Z$ a continuous mapping in the first variable, $0\le\alpha<\omega_1$; let $((\varphi_{i,n}:i\in
  I_n))_{n=1}^\infty$ be a sequence of locally finite partitions of unity on $X$ and $((x_{i,n}:i\in I_n))_{n=1}^\infty$ a sequence of families of points $x_{i,n}\in X_{B_\alpha}(f)$ such that (\ref{eq:3}) holds, and let the sequence $(\mathcal A_n)_{n=1}^{\infty}$ of families $\mathcal A_n=({\rm supp}\varphi_{i,n}:i\in I_n)$  be uniformly pointwise finite in $X$. Then \mbox{$f\in B_{\alpha+1}(X\times Y,Z)$.}
\end{theorem}

\begin{proof}
For every $n\in {\mathbb N}$ we well order the set  $I_n$ and for all \mbox{$(x,y)\in X\times Y$} let $$ f_n(x,y)={\sum\limits_{i\in I_n}}^\lambda\varphi_{i,n}(x)f(x_{i,n},y). $$ Analogously as in the proof of Theorem~\ref{th:4.1} we have $f_n\in {B}_\alpha
(X\times Y, Z)$ for every $n\in\mathbb N$.

It remains to prove that  $f_n(x,y)\to f(x,y)$ on $X\times Y$. Fix $(x_0,y_0)\in X\times Y$ and  consider a neighborhood $W$ of $z_0=f(x_0,y_0)$ in  $Z$. Denote by $m$ the number $k_{x_0}$ from Definition~\ref{def:3}. Since $\lambda_{m}$ is continuous and $\lambda_m(z_0,\dots, z_0,\alpha_1,\dots, \alpha_m)=z_0$ for all $(\alpha_1,\dots, \alpha_m)\in S_m$ and $S_m$ is compact, there exists a neighborhood $W_1$ of $z_0$ in $Z$ such that
$\lambda_m(z_1,\dots, z_m,\alpha_1,\dots, \alpha_m)\in W$ for all $z_1,\dots,z_m\in W_1$ and $(\alpha_1,\dots, \alpha_m)\in S_m$. The continuity of $f$ in the first variable at $(x_0,y_0)$ implies the existence of a neighborhood $U$ of $x_0$ in $X$ such that $f(x,y_0)\in W_1$ for every $x\in U$. According to~(\ref{eq:3}) there is a number $n_0$ such that  $x_{i,n}\in U$ for all  $n\geq n_0$ and $i\in I_n$ with $x_0\in {\rm supp}\varphi_{i,n}$.

We will show that $f_n(x_0,y_0)\in W$ for all $n\geq n_0$. Take $n\geq n_0$ and set $$I=\{i\in I_n:x_0\in {\rm supp}\varphi_{i,n}\}.$$
Let $I=\{i_1, i_2,\dots i_k\}$, where $i_1< i_2<\dots < i_k$ and $k\leq m$. Denote $\alpha_1=\varphi_{i_1,n}(x_0),\dots,\alpha_k=\varphi_{i_k,n}(x_0)$, $z_1=f(x_{i_1,n},y_0),\dots , z_k=f(x_{i_k,n},y_0)$. Then  $z_1,\dots z_k \in W_1$. Taking into account Proposition~\ref{pr:2.1} $(i)$, we have $$f_n(x_0,y_0)=\lambda_k(z_1,\dots , z_k,\alpha_1,\dots \alpha_k)=$$$$= \lambda_m(z_1,\dots , z_k,z_0,\dots,z_0,\alpha_1,\dots \alpha_k,0,\dots, 0)\in W.$$\hfill$\Box$
\end{proof}

A topological space $X$ is {\it strongly countably dimensional} if there exists a sequence $(X_n)_{n=1}^{\infty}$ of sets $X_n\subseteq X$ such that $X=\bigcup\limits_{n=1}^{\infty}X_n$ and ${\rm dim} X_n< n$ for every $n\in\mathbb N$, where by ${\rm dim\,}Y$ we denote \v{C}ech-Lebesgue dimension of $Y$ (see \cite[p.~564]{Eng}).

\begin{theorem}\label{th:4.3} Let $X$ be a Hausdorff paracompact strongly countably dimensional strong PP-space, $Y$ a topological space, $(Z,\lambda)$ an equiconnected space and  $0\le\alpha<\omega_1$. Then $(X,Y,Z)$ is a Rudin $\alpha$-triple. \end{theorem}

\begin{proof}
Let $f\in C\overline{B}_\alpha
(X\times Y, Z)$. Using Definition~\ref{def:2}, we choose a sequence $(\varphi_{i,n}:i\in I_n))_{n=1}^\infty$ of locally finite partitions of unity $(\varphi_{i,n}:i\in  I_n)$ on $X$ and a sequence $((x_{i,n}:i\in I_n))_{n=1}^\infty$ of families of points $x_{i,n}\in X_{B_\alpha}(f)$ such that (\ref{eq:3}) holds. By \cite[Theorem~5.1.10]{Eng1} for every $n\in\mathbb N$ there exists a locally finite cover $(U_{j,n}:j\in J_n)$ of $X$ which refines $({\rm supp}\,\varphi_{i,n}:i\in I_n)$ and such that for every  $m\in\mathbb N$ and $x\in X_m$ there is a neighborhood $U$ of $x$ with \mbox{$\bigl|\{j\in J_n:U\cap U_{j,n}\ne\emptyset\}\bigr|\leq m$.} For every $n\in\mathbb N$ choose a locally finite partition of unity $(\psi_{j,n}:j\in
  J_n)$ on $X$ subordinated to $(U_{j,n}:j\in J_n)$. For every $j\in J_n$ denote by $u_{j,n}$ such an element $x_{i,n}$ that $U_{j,n}\subseteq {\rm supp}\,\varphi_{i,n}$. Let us observe that the sequences $((\psi_{j,n}:j\in  J_n))_{n=1}^\infty$ and $((u_{j,n}:j\in J_n))_{n=1}^\infty$ satisfy the conditions of Theorem~\ref{th:4.2}, hence \mbox{$f\in B_{\alpha+1}(X\times Y,Z)$.}\hfill$\Box$
\end{proof}

The following corollary generalizes Theorem~\ref{B}.

\begin{corollary}\label{cor:5.1} Let $X$ be a strongly countably dimensional metrizable space, $Y$ a topological space, $(Z,\lambda)$ an equiconnected space. Then $(X,Y,Z)$ is a Rudin $\alpha$-triple for every $0\le \alpha<\omega_1$. \end{corollary}

\section{The case $\alpha>0$}\label{sec:4}

\begin{definition}\label{def:5} {\rm A topological space $X$ is {\it contractible} if there exist a point $x^*\in X$ and a continuous function $\gamma:X\times
[0,1]\to X$ such that $\gamma(x,0)=x$ and $\gamma(x,1)=x^*$ for every $x\in X$.  A contractible space $X$ with such a function $\gamma$ and such a point $x^*$ we denote by $(X,\gamma, x^*)$.} \end{definition}

Clearly, every equiconnected space $(X,\lambda)$ is contractible.

\begin{proposition}\label{pr:7.1}
  Let $X$ and $Y$ be topological spaces, $(Z,\gamma, z^*)$ a contractible space, $0\le\alpha<\omega_1$, $(g_i)_{i\in I}$ a family of mappings \mbox{$g_i\in B_\alpha(Y,Z)$} and let $(\varphi_i)_{i\in I}$ be a family of continuous functions $\varphi_i:X\to [0,1]$ such that the family  $({\rm supp}\,\varphi_i)_{i\in I}$ is discrete in $X$. Then the mapping $f:X\times Y\to Z$,
  $$
  f(x,y)=\left\{\begin{array}{ll}
                  \gamma(g_{i}(y),1-\varphi_{i}(x)), & \mbox{if}\,\,\, x\in {\rm supp}\,\varphi_i \,\,\,\mbox{for some}\,\, i\in I, \\
                  z^*, & \mbox{else},
                \end{array}
  \right.
  $$
  is of the Baire class $\alpha$.
\end{proposition}

\begin{proof}  We first consider the case when $\alpha=0$.
Fix a point $x_0\in X$. Since the family $({\rm supp}\,\varphi_i)_{i\in I}$ is discrete, we can choose a neighborhood $U_0$ of $x_0$ such that  $\{i\in I: U_0\cap{\rm supp}\,\varphi_i\ne\emptyset\}\subseteq \{j\}$ for some $j\in I$. Then $f(x,y)=\gamma(g_{j}(y),1-\varphi_{j}(x))$ for all $x\in U_0$ and $y\in Y$. Hence, $f$ is continuous at every point of $\{x_0\}\times Y$.

Further, we use the induction on $\alpha$ analogously as in the proof of Theorem~\ref{th:2.3}.\hfill$\Box$
\end{proof}

Recall that a subset $A$ of a topological space $X$ is {\it a zero (co-zero) set} if there exists a continuous function $f:X\to [0,1]$ such that $A=f^{-1}(0)$ ($A=f^{-1}((0,1])$). A set $A$ is {\it functionally ambiguous} if there are two sequences $(A_n)_{n=1}^\infty$ and $(B_n)_{n=1}^\infty$  of zero-sets $A_n$ and co-zero sets $B_n$ such that $A=\bigcup\limits_{n=1}^\infty A_n=\bigcap\limits_{n=1}^\infty B_n$.

\begin{lemma}\label{lemma1.7}
Let $X$ be a topological space and ${\mathcal G}$ be a locally finite cover of $X$ by co-zero sets. Then there exists a disjoint locally finite cover of  $X$ by functionally ambiguous sets which refines ${\mathcal G}$ .
\end{lemma}

\begin{proof}
Write \mbox{${\mathcal G}=\{G_\alpha:0\le \alpha<\beta\}$}, where $\beta$ is an ordinal. Set \mbox{$A_0=G_0$}. For every $0<\alpha<\beta$ let
\mbox{$A_\alpha=G_\alpha\setminus \bigcup\limits_{\xi<\alpha} G_\xi$}. Since $\mathcal G$ is locally finite, the set $\bigcup\limits_{\xi<\alpha} G_\xi$ is a co-zero set. Then $A_\alpha$ is functionally ambiguous set as the difference of co-zero sets. Clearly, the family ${\mathcal A}=(A_\alpha:0\le\alpha<\beta)$ is to be found.\hfill$\Box$ \end{proof}

\begin{definition}{\rm
 A topological space  $X$ is said to be  {\it weakly collectionwise normal} if for an arbitrary discrete family $(F_s:s\in S)$ of zero sets in $X$
 there exists a discrete family $(U_s:s\in S)$ of co-zero sets in $X$ such that $F_s\subseteq U_s$ for all $s\in S$.}
\end{definition}

Let us observe that every collectionwise normal space is weakly collectionwise normal. Remark also that every metrizable space or, more general, a space equipped with the topology generated by a pseudo-metric, is weakly collectionwise normal.

\begin{proposition}\label{pr:21} Let $0<\alpha<\omega_1$, let $X$ be a weakly collectionwise normal space, $Y$ a topological space, $Z$ a contractible space,
$(X_s:s\in S)$ a disjoint locally finite cover of  $X$ by functionally ambiguous sets, $g_s\in B_\alpha(Y,Z)$ for every $s\in S$ and
$f(x,y)=g_s(y)$ if $(x,y)\in X_s\times Y$. Then $f\in B_\alpha(X\times Y,Z)$.
\end{proposition}

\begin{proof} For every $s\in S$ there exists an increasing sequence  $(F_{s,n})_{n=1}^\infty$ of zero sets such that
$X_s=\bigcup\limits_{n=1}^\infty F_{s,n}$. Since for every $n\in\mathbb N$ all the sets $F_{s,n}$ are closed, every family $(F_{s,n}:s\in S)$ is discrete in $X$.
Taking into account that $X$ is weakly collectionwise normal, for every $n\in\mathbb N$ we choose a discrete family $(U_{s,n}:s\in S)$ of co-zero sets in $X$ such that $F_{s,n}\subseteq U_{s,n}$ for all $s\in S$. Let $\varphi_{s,n}\in C(X,[0,1])$ be such functions that $F_{s,n}=\varphi_{s,n}^{-1}(1)$ and $U_{s,n}=\varphi_{s,n}^{-1}((0,1])$ for all $(s,n)\in S\times\mathbb N$.

 We choose an increasing sequence  $(\alpha_n)_{n=1}^\infty$ of ordinals $\alpha_n<\alpha$ such that $\alpha=\lim\limits_{n\to\infty}(\alpha_n+1)$ and for every $s\in S$ choose a sequence of functions  $g_{s,n}\in B_{\alpha_n}(Y,Z)$ which is pointwise convergent to $g_s$ on $Y$.

 Let $\gamma\in C(Z\times [0,1],Z)$ and $z_0\in Z$ be such that $\gamma(z,0)=z$ and $\gamma(z,1)=z_0$ for all $z\in Z$. Define a function $f_n:X\times Y\to Z$ in the following way:
  $$
  f_n(x,y)=\left\{\begin{array}{ll}
                  \gamma(g_{s,n}(y),1-\varphi_{s,n}(x)), & \mbox{if}\,\,\, x\in U_{s,n} \,\,\,\mbox{for some}\,\, s\in S, \\
                  z_0, & \mbox{else}.
                \end{array}
  \right.
  $$
    According to Proposition~\ref{pr:7.1},  $f_n\in B_{\alpha_n}(X\times Y,Z)$ for every $n\in\mathbb N$.

  It remains to prove that $f_n\to f$ on $X\times Y$. Indeed, let $(x,y)\in X\times Y$. Then there exists $s\in S$ such that $x\in X_s$. Since the sequence   $(F_{s,n})_{n=1}^\infty$ increases, there exists a number $n_0$ such that  $x\in F_{s,n}$ for all $n\ge n_0$. Then $f_n(x,y)=\gamma(g_{s,n}(y),0)=g_{s,n}(y)$ for all $n\ge
  n_0$. Hence, $$f_n(x,y)\to g_s(y)=f(x,y).$$ Therefore, $f\in B_\alpha(X\times Y,Z)$.\hfill$\Box$
\end{proof}

\begin{theorem}\label{th:6.6}
  Let $X$ be a strong PP-space, $Y$ a topological space, $Z$ a contractible space and $0<\alpha<\omega_1$. Then $(X,Y,Z)$ is a Rudin $\alpha$-triple.
\end{theorem}

\begin{proof}
  Let $f\in  C\overline{B}_\alpha(X\times Y,Z)$. For the set $X_{B_\alpha}(f)$ there exist a sequence $((\varphi_{i,n}:i\in
  I_n))_{n=1}^\infty$ of locally finite partitions of unity on $X$ and a sequence $((x_{i,n}:i\in I_n))_{n=1}^\infty$
  of families of points of $X_{B_\alpha}(f)$ such that condition~(\ref{eq:3}) holds.

  In accordance with \cite[Proposition~3.2]{M} there exists a pseudo-metric on $X$ such that all the functions $\varphi_{i,n}$ are continuous with respect to this pseudo-metric. Denote by $\mathcal T$ the topology on $X$ generated by the pseudo-metric. Obviously, the topology $\mathcal T$ is weaker than the initial one and all the sets ${\rm supp}\,\varphi_{i,n}$ are $\mathcal T$-open.
  Using the paracompactness of $(X,{\mathcal T})$, for every $n$ we choose a locally finite open cover ${\mathcal V}_n$ which refines the open cover \mbox{$({\rm supp}\,\varphi_{i,n}:i\in I_n)$.} By Lemma~\ref{lemma1.7} for every $n$ there exists a disjoint locally finite cover ${\mathcal A}_n=(A_{s,n}:s\in S_n)$ of
  $(X,{\mathcal T})$ by ambiguous sets which refines ${\mathcal V}_n$. For every $s\in S_n$ we choose any $i(s)\in I_n$ such that $A_{s,n}\subseteq
  {\rm supp}\,\varphi_{i(s),n}$. For all $n$ and $(x,y)\in X\times Y$ let
  $$
  f_n(x,y)=f^{x_{i(s),n}}(y),\,\,\mbox{if}\,\, x\in A_{s,n}.
  $$
  Since $f^{x_{i(s),n}}\in B_\alpha(Y,Z)$ for every $s\in S_n$ and $(X,{\mathcal T})$ is weakly collectionwise normal, Proposition~\ref{pr:21} implies that $f_n\in B_\alpha((X,{\mathcal T})\times Y,Z)\subseteq B_\alpha(X\times Y,Z)$ for every $n\in\mathbb N$.

  Now we   show that $f_n\to f$ on $X\times Y$. Fix $(x,y)\in X\times Y$. There exists a sequence $(s_n)_{n=1}^\infty$, $s_n\in S_n$, such that
 $x\in A_{s_n}\subseteq {\rm supp}\,\varphi_{i(s_n),n}$. Condition~(\ref{eq:3}) implies that $x_{i(s_n),n}\to x$. Since  $f$ is continuous in the first variable,
 $$
 f_n(x,y)=f({x_{i(s_n),n}},y)\to f(x,y).
 $$
  Hence, $f\in B_{\alpha+1}(X\times Y,Z)$.\hfill$\Box$
\end{proof}

\section{Vertically nearly separately continuous mappings which do not belong to the first Baire class}\label{sec:5}

Let $X=\{0\}\cup\bigcup\limits_{n=1}^\infty X_n$, where $X_n=\{\frac 1n\}\cup\bigcup\limits_{m=n^2}^\infty \{\frac{1}{n}+\frac{1}{m}\}$. We define a topology on $X$ in the following way. All points of the form $\frac 1n+\frac 1m$ will be isolated points of $X$. The base of neighborhoods of a point $\frac 1n$ are the sets of the form $X_n\setminus\bigcup\limits_{m=n^2}^k\{\frac 1n+\frac 1m\}$, $k=n^2,n^2+1,\dots$. Finally, as neighborhoods of $0$ we take all the sets obtained from  $X$ by removing a finite number of $X_n$'s and a finite number of points of the form $\{\frac 1n+\frac 1m\}$ in all the remaining $X_n$'s. The space $X$ with this topology is a sequential space that is not a Fr\'{e}chet space~\cite[Example 1.6.19]{Eng}.

The following example is a development of Example~5.7 from \cite{BanakhT} and shows that the condition that $X$ is strong PP-space is essential in Theorems~\ref{th:4.1},\ref{th:4.3} and \ref{th:6.6}. It then follows that Theorem~\ref{E} cannot be generalized to Rudin triples.

\begin{example}\label{ex:1}
  Let $0\le\alpha<\omega_1$, let $Y$ and $Z$ be topological spaces such that \mbox{$B_{\alpha+2}(Y,Z)\setminus B_{\alpha+1}(Y,Z)\ne\O$}. Then there exists a function \mbox{$f\in C\overline{B}_\alpha(X\times Y,Z)$}, which is not a pointwise limit of a sequence of functions $f_n\in CB_\alpha(X\times Y,Z)$.
\end{example}

\begin{proof}
 Let $g\in B_{\alpha+2}(Y,Z)\setminus B_{\alpha+1}(Y,Z)$. Then there exist a sequence $(g_n)^{\infty}_{n=1}$ of mappings $g_n\in B_{\alpha+1}(Y,Z)$ and a family $(g_{nm}:n,m\in\mathbb N)$ of mappings $g_{nm}\in B_\alpha(Y,Z)$ such that $g_n(y)\to g(y)$ for all $y\in Y$ and $g_{nm}(y)\to g_n(y)$, when $m\to\infty$, for all $n\in\mathbb N$ and $y\in Y$. Let $K=\{(n,m):n\in\mathbb N, m\in\mathbb N, m\geq n^2\}$, $x_0=0$, $x_n=\frac 1n$, $x_{nm}=\frac 1n+\frac 1m$, where $(n,m)\in K$, and consider the function $f:X\times Y\to Z$,
$$
f(x,y)=\left\{\begin{array}{ll}
                g(y), & x=x_0, \\
                g_n(y), & x=x_n, \\
                g_{nm}(y), & x=x_{nm}.
              \end{array}
\right.
$$
It is easy to verify that  $X_{B_\alpha}(f)=\{x_{nm}:(n,m)\in K\}$, $\overline{X_{B_\alpha}(f)}=X$ and $f_y\in C(X,Z)$ for every $y\in Y$. Hence, $f\in C\overline{B}_\alpha(X\times Y,Z)$.

Assume that there exists a sequence of functions $f_n\in CB_\alpha(X\times Y,Z)$ which pointwise converges to $f$. Then $f_n(0,y)\to f(0,y)=g(y)$ for all $y\in Y$, i.e. $g\in B_{\alpha+1}(Y,Z)$, a contradiction.\hfill$\Box$
 \end{proof}

\begin{lemma}\label{lemma8.1}

Let $X$ be a topological space, $a\in \mathbb R$, let $F, H$ be zero sets in $X$, $G$ a co-zero set in $X$ such that $H\subseteq F\cap G$. Then there exists a function $f:X\times \mathbb R\to\mathbb R$ with the followings properties:

$(1)$ $f$ is continuous in the first variable;

$(2)$ $f$ is jointly continuous on $(x,y)\in (X\setminus F)\times \mathbb R$;

$(3)$ $f(x,y)=0$ if $(x,y)\in (X\setminus G)\times \mathbb R$ or $(x,y)\in F\times (\mathbb R\setminus\{a\})$;

$(4)$ $f(x,a)=1$ if $x\in H$.
\end{lemma}

\begin{proof}
We pick continuous functions $\varphi:X\to [0,1]$ and $\psi:X\to[0,1]$ such that $H=\varphi^{-1}(1)$, $X\setminus G=\varphi^{-1}(0)$ and $F=\psi^{-1}(0)$. Moreover, consider the continuous function $g:\mathbb R\times (0,+\infty)\to [0,1]$,
$$
g(u,v)=\left\{\begin{array}{ll}
                \cos\frac{\pi (u-a)}{2v}, & |u-a|\leq v, \\
                0, & |u-a|> v.
              \end{array}
\right.
$$

Now consider the continuous function $f:X\times \mathbb R\to\mathbb R$,
$$
f(x,y)=\left\{\begin{array}{ll}
                \varphi(x)g(y,\psi(x)), & x\in X\setminus F, \\
                \varphi(x)\chi_{\{a\}}(y), & x\in F,
              \end{array}
\right.
$$
where by $\chi_{\{a\}}$ we denote the characteristic function of the set $\{a\}$. Let us observe that $(1)$ and $(2)$ follow from the continuity of functions  $\varphi$, $\psi$ and $g$ and the equality $\lim\limits_{v\to 0}g(u,v)=\chi_{\{a\}}(u)$ for every $u\in\mathbb R$. Conditions $(3)$ and $(4)$ immediately follow from the definition of $f$.\hfill$\Box$
\end{proof}

Denote by $\mathbb R^\infty$ the collection of all sequences with finite support, i.e. sequences of the form $(\xi_1,\xi_2,\dots,\xi_n,0,0,\dots)$, where $\xi_1,\xi_2,\dots,\xi_n\in\mathbb R$. Clearly, $\mathbb R^\infty$ is a linear subspace of the space $\mathbb R^{\mathbb N}$ of all sequences. Denote by $E$ the set of all sequences $e=(\varepsilon_n)_{n=1}^\infty$ of positive reals $\varepsilon_n$ and let $$U_e=\{x=(\xi_n)_{n=1}^\infty\in\mathbb R^\infty: (\forall n\in\mathbb N)(|\xi_n|\le \varepsilon_n)\}.$$
We consider on $\mathbb R^\infty$ the topology in which the system $\mathcal U_0=\{U_e:e\in E\}$ forms the base of neighborhoods of zero.

The following example shows that the analog of Theorem~\ref{th:4.3} is not true for strongly countably dimensional locally convex Hausdorff paracompact spaces.

\begin{example}\label{ex:2}
  Let $X=\mathbb R^\infty$ and $Y=Z=\mathbb R$. Then there exists a function \mbox{$f\in C\overline{C}(X\times Y,Z)$} which is not a pointwise limit of a sequence of separately continuous functions.
\end{example}

\begin{proof}
 For every $n\in\mathbb N$ let $$F_n=\{(\xi_1,\dots,\xi_n,0,0\dots)\in X: \max\{|\xi_1|,\dots,|\xi_n|\}\leq\frac{1}{n}\},$$
 $$\tilde{F}_n=\{(\xi_1,\dots,\xi_n,\xi_{n+1},\dots)\in X: \max\{|\xi_1|,\dots,|\xi_n|\}\leq\frac{1}{n}\},$$
 $$G_n=\{(\xi_1,\dots,\xi_n,\xi_{n+1},\dots)\in X: \max\{|\xi_1|,\dots,|\xi_n|\}<\frac{1}{n-\frac{1}{2}}\}$$
 and $$H_n=\bigcap\limits_{m=n}^{\infty}\left(\bigcup\limits_{k=n}^{m}F_k\cup \tilde{F}_m\right).$$
 It is easily seen that all the sets $H_n$ are zero sets, and all the sets $G_n$ are co-zero sets.

Let $\mathbb Q=\{r_n:n\in\mathbb N\}$, where $r_n\ne r_m$ for $n\ne m$. Fix $n\in\mathbb N$. For the space $X$, a number $a=r_n$ and for sets $F=H_1$, $H_n$ and $G_n$ there exists a function $f_n:X\times \mathbb R\to\mathbb R$ which satisfies $(1)-(4)$ of Lemma~\ref{lemma8.1}. We set $f(x,y)=\sum\limits_{n=1}^{\infty}f_n(x,y)$. We will show that \mbox{$f\in C\overline{C}(X\times Y,Z)$}, but $f$ is not a pointwise limit of a sequence of separately continuous functions.

Remark that $f_n(x,y)=\chi_{r_n}(y)$ for all $n\in\mathbb N$, $y\in Y$ and $x\in H_n$. Then for the point $x_0=(0,0,\dots)\in\bigcap\limits_{n=1}^{\infty}H_n$ the function $f^{x_0}$ is the Dirichlet function, which is not of the first Baire class. Hence, $f$ is not a pointwise limit of a sequence of separately continuous functions.

Now we prove that $X\setminus F\subseteq X_C(f)$. Let $x_1\in X\setminus F$. Since $x_1\ne x_0$, there exist a neighborhood $U\subseteq X\setminus F$ of $x_1$ and a number $n_0\in\mathbb N$ such that  $U\cap G_n=\emptyset$ for all $n\geq n_0$. Then, taking into account condition~$(3)$, for every $x\in U$ we have $f(x,y)=\sum\limits_{k=1}^{n_0}f_k(x,y)$ for all $y\in Y$. By condition~$(2)$, $f$ is continuous on $U\times Y$, in particular, $x_1\in X_C(f)$. Since $\overline{X\setminus F}=X$, $\overline{X_C(f)}=X$.

It remains to show that $f$ is continuous in the first variable. Fix $y_0\in Y$. Since the sequence $(G_n)_{n=1}^{\infty}$ is decreasing and $\{x_0\}=\bigcap\limits_{n=1}^{\infty}\overline{G_n}$, conditions  $(3)$ and $(1)$ imply that the function $f_{y_0}$ is continuous at all points $x\ne x_0$. Notice that a given function  $g:\mathbb R^\infty\to\mathbb R$ is continuous if and only if the restriction of $g$ to every space $X_n=\{(\xi_1,\dots,\xi_n,0\dots):\xi_1,\dots, \xi_n\in\mathbb R\}$ is continuous. Therefore, it is sufficient to show that every restriction of  $f_{y_0}$ to $X_n$ is continuous at $x_0$. Choose a number $m\in\mathbb N$ such that  $y_0\not\in\{r_k:k\geq m\}$. Then by condition $(3)$, $f(x,y_0)=\sum\limits_{k=1}^{m}f_k(x,y_0)$ for all $x\in F$. Since for every $n\in\mathbb N$ the set $F_n=F\cap X_n$ is a neighborhood of  $x_0$ in $X_n$, one has that $(1)$ imply the continuity of the restriction of $f_{y_0}$ on $X_n$ at the point $x_0$.\hfill$\Box$
 \end{proof}

 Remark that the spaces from Example~\ref{ex:1} and Example~\ref{ex:2} are PP-spaces. But combining these examples and Theorem~\ref{th:4.1} implies that they are not strong PP-spaces.

Let us observe that the inclusion \mbox{$C\overline C(X\times Y, \mathbb R)\subseteq B_{1}(X\times Y, \mathbb R)$} for any topological space $Y$ does not imply that $X$ is a strong PP-space. For example, the inclusion holds for any topological space $X$ on which every continuous real-valued function  is constant. These spaces, in general, need not by regular (see \cite[p.~119]{Eng}) and (strong) $PP$-spaces. Therefore, it is natural to ask the following question.

\begin{question}\label{qu8.4}
Does there exist a completely regular space $X$ which is not a (strong) PP-space and \mbox{$C\overline C
(X\times Y, \mathbb R)\subseteq B_{1}(X\times Y, \mathbb R)$} for any topological space $Y$?
\end{question}


{\small
\begin{thebibliography}{10}

\bibitem{BanakhT} T. Banakh, {\it (Metrically) quarter-stratifiable spaces
and their applications}, Math. Studii {\bf 18}(1) (2002), 10--28.

\bibitem{Bu1} M. Burke, {\it Borel measurability of separately continuous
functions}, Topology Appl. {\bf 129} (1) (2003), 29--65.

\bibitem{Eng} R. Engelking, {\it General Topology}, Revised and completed edition, Heldermann Verlag, Berlin (1989).

\bibitem{Eng1} R. Engelking,  {\it  Theory of dimensions, finite and infinite}, Revised and completed edition. Heldermann Verlag, Lemgo (1995).

\bibitem{Hahn} H. Hahn,  {\it Reelle Funktionen.1.Teil.
Punktfunktionen}, Leipzig: Academische Verlagsgesellscheft M.B.H. (1932).

\bibitem{KMas3} A. Kalancha, V. Maslyuchenko, {\it \v{C}ech-Lebesgue dimension and Baire classification of vector-valued separately continuous mappings}, Ukr.Math.J. {\bf 55} (11) (2003), 1596--1599. (in Ukrainian)

\bibitem{K4} O. Karlova, {\it Baire classification of mappings which are continuous in the first variable and of the functional class  $\alpha$ in the second one},  Math. Bull. NTSH. {\bf 2} (2005), 98--114.


\bibitem{Leb1} H. Lebesgue,  {\it Sur l'approximation des fonctions},
Bull. Sci. Math. {\bf 22} (1898), 278--287.


\bibitem{Moran} W. Moran,  {\it Separate continuity and supports of
measures}, J. London Math. Soc. {\bf 44} (1969), 320--324.

\bibitem{M} V. Mykhaylyuk, {\it Baire classification of separately continuous functions and Namioka property}, Ukr. Math. Bull. {\bf 5} (2) (2008), 203--218.

\bibitem{Ru} W. Rudin,  {\it Lebesgue first theorem}, Math. Analysis
and Applications, Part B. Edited by Nachbin. Adv. in Math.
Supplem. Studies {\bf 78}, Academic Press (1981), 741--747.

\bibitem{Sob2} O. Sobchuk, {\it Baire classification and Lebesgue spaces}, Bull. of Chernivtsi Nat. Univ. Mathematics, {\bf 111} (2001), 110--112.

\bibitem{SobPP} O. Sobchuk, {\it $PP$-spaces and Baire classification}, International Conference on Functional Analysis and its
Applications, dedicated to the 110th anniversary of Stefan Banach.
Book of abstracts (2002), 189.

\bibitem{Vera} G. Vera, {\it Baire mesurability of separately
continuous functions},  Quart. J. Math. Oxford {\bf 39} (2) (1988), 109--116.











\end{thebibliography}

}
\end{document}